\documentclass[preprint,12pt]{elsarticle}

\usepackage{amsmath,amssymb, amscd, amsthm}
\usepackage{accents}
\usepackage{setspace}
\usepackage{fancyhdr}
\usepackage{graphicx}
\usepackage{subfigure}
\usepackage{psfrag}
\usepackage{titlesec}
\usepackage{epstopdf}
\usepackage{textcomp}
\usepackage{xspace}
\usepackage{booktabs}
\usepackage{multirow}
\usepackage{tabu}
\usepackage{enumerate}
\usepackage{tikz}
\usepackage{tikz-cd}
\usepackage{xcolor}

\newtheorem{thm}{Theorem}[section]

\newtheorem{cor}[thm]{Corollary}
\newtheorem{lem}[thm]{Lemma}
\newtheorem{prop}[thm]{Proposition}

\newtheorem{defn}[thm]{Definition}

\newproof{pf}{Proof}

\newcommand{\Hom}{\text{Hom}}
\newcommand{\coker}{\text{coker }}
\newcommand{\Ass}{\text{Ass}}

\newcommand{\pd}{\text{pd}}

\newcommand{\Ext}{\text{Ext}}

\newcommand{\codim}{\text{codim}}

\newcommand{\Ann}{\text{Ann}}

\providecommand{\selectlanguage}[1]{\relax}

\newcommand{\minus}{\smallsetminus}

\journal{Journal of Pure and Applied Algebra}

\begin{document}

\begin{frontmatter}



\title{An extension of a theorem \\ of Frobenius and Stickelberger \\ to modules of projective dimension one \\ over a factorial domain}%


\author[label1]{Joseph P. Brennan\corref{cor1}}
\ead{joseph.brennan@ucf.edu}

\author[label1]{Alexander York}
\ead{a.york@ucf.edu}

\address[label1]{Department of Mathematics, University of Central Florida\\ 4393 Andromeda Loop N, Orlando, FL 32816}
\cortext[cor1]{Corresponding author}

\begin{abstract}
Let $R$ be a Noetherian ring.   A quasi-Gorenstein \(R\)-module is an \(R\)-module such that the grade of the module and the projective dimension of the module are equal and the canonical module of the module is isomorphic to the module itself.  After discussing properties of finitely generated quasi-Gorenstein modules, it is  shown that this definition allows for a characterization of diagonal matrices of maximal rank over a commutative Noetherian factorial domain \(R\) extending a theorem of Frobenius and Stickelberger to modules of projective dimension 1 over a commutative Noetherian factorial domain.
\end{abstract}


\end{frontmatter}



\section{Introduction}

This paper brings together two threads in the theory of modules that have not experienced extensive development in recent times.  The first of which is the study of diagonalization of matrices under equivalence over arbitrary commutative rings.  Frobenius and Stickelberger in their famous paper \cite{Frobenius1879} classified finitely generated modules over Abelian groups and an extension of this classifcation theorem to modules over principal ideal domains \cite{Jacobson} brings us to the end of that development.  Any further attempts have been thwarted by some basic ideas and properties in more general commutative rings.  One of the simplest examples of an obstruction to extending the results for modules of projective dimension one is the following matrix over \(\mathbb{Z}[x]\)

\begin{align}
\begin{pmatrix}
2 & x \\
0 & 3
\end{pmatrix}.
\end{align}

It is easily seen that this matrix is not equivalent to a diagonal matrix.  As \(\mathbb{Z}[x]\) is usually the first ring looked at after a principal ideal domain is considered, the hope for a general theorem to extend diagonlization to all matrices has stopped here.\newline

The second thread is the question of when a matrix over a commutative ring is equivalent to its transpose.  The theorem of Frobenius and Stickelberger \cite{Frobenius1879} demonstrates that this property holds over \(\mathbb{Z}\) and the classical extension to principal ideal domains shows that principal ideal domains have this property as well.  Alas, the matrix above is also an obstruction to an extension of this result to more general commutative rings.  It can easily be seen that over \(\mathbb{Z}[x]\) the transpose of \((1.1)\) is not equivalent to \((1.1).\)\newline

So the questions then become "What matrices are diagonalizable in more general rings?" and "What matrices are equivalent to their transpose in more general rings?".  The second question can easily be answered, but the first requires a bit more finesse.\newline

In the rest of the paper all modules are assumed to be finitely generated over \(R\) and \(R\) will be a Noetherian factorial domain.  The principal reason for restricting our interest to such rings is that in a factorial domain every prime ideal of codimension 1 is a principal ideal, and we can compare associated primes of modules with their Ext modules in a relatively simple way, [Prop 3.2.2 in \cite{VasCMCAAG}].\newline

The main result of this paper deals with the first question posed above, however, the second question leads to some interesting definitions and ideas used extensively for the first.  We can see this if we consider a module, \(M,\) of projective dimension one presented by a \(n\times n\)-matrix \(m.\) It has a resolution

\[
\begin{tikzcd}
0 \arrow[r] & R^n \arrow[r, "m"] & R^n \arrow[r] & M \arrow[r] & 0.
\end{tikzcd}
\]

To introduce the transpose matrix \(m^T\) of \(m\) we apply \(\Hom_R(-,R)\) to the sequence to obtain

\[
\begin{tikzcd}
0 \arrow[r] & \Hom_R(M,R) \arrow[r] & R^n \arrow[r, "m^T"] & R^n \arrow[r] & \Ext_R^1(M,R) \arrow[r] & 0
\end{tikzcd}
\]

It is then clear that if \(m\) is equivalent to \(m^T\) we get a commutative diagram of the form

\[
\begin{tikzcd}
& 0 \arrow[r] & R^n \arrow[r, "m"] \arrow[d, "\cong"'] & R^n \arrow[r] \arrow[d, "\cong"'] & M \arrow[r] & 0 \\
0 \arrow[r] & \Hom_R(M,R) \arrow[r] & R^n \arrow[r, "m^T"] & R^n \arrow[r] & \Ext_R^1(M,R) \arrow[r] & 0
\end{tikzcd}
\]

\noindent which shows that \(\Hom_R(M,R)=0\) and \(M\cong \Ext_R^1(M,R).\)  This indicates the importance of such modules M, which are known in the literature as quasi-Gorenstein modules.\newline

Returning to the first question, in order for an \(R\)-module \(M\) to be presented by a diagonal matrix and hence be a direct sum of cyclic $R$-modules with principal annihilators, the following are necessary conditions:
\bigskip

\begin{itemize}
\item[\((a)\)] \(M\) must be a module of projective dimension 1.
\bigskip
\item[\((b)\)] The matrix presenting \(M\) is equivalent to its transpose (as in the diagrams above).
\bigskip
\item[\((c)\)] \(M\) must be filtered by \textit{quasi-Gorenstein} submodules whose quotients are cyclic and have principal annihilators.\end{itemize}

\bigskip

The reader should note that none of these conditions is sufficient. Matrix \((1.1)\) shows that \((a)\) is not sufficient.  To see that \((b)\) and \((c)\) are not sufficient consider the following matrix over \(k[x,y]:\)

\begin{align}
m=\begin{pmatrix}
x & y \\ 0 & x
\end{pmatrix}
\end{align}

This matrix is of full rank and it can be shown that this matrix satisfies \((b).\)  The associated prime of the module \(M=\coker\!(m)\) is \((x)\) and \(\text{Ann}(M)=(x^2).\)  The two filtrations

\[
\begin{tikzcd}
0 \arrow[r] & k[x,y]/(x) \arrow[r] & M \arrow[r] & k[x,y]/(y) \arrow[r] & 0\\
0 \arrow[r] & k[x,y]/(x^2) \arrow[r] & M \arrow[r] & k \arrow[r] & 0
\end{tikzcd}
\]

\noindent show that \(M\) has no filtration of the correct form.

The following theorem characterizes diagonalizable matrices.

\begin{thm}\label{Rthm1} Let \(R\) be a commutative factorial domain and \(M\) an \(R\)-module of projective dimension one presented by a full rank \(n\times n\) matrix \(m\) then the following conditions are equivalent:

\bigskip

\begin{itemize}
\item[(i)] The module \(M\) is quasi-Gorenstein and has a minimal cyclic-filtration consisting of quasi-Gorenstein submodules.
\bigskip
\item[(ii)] The matrix \(m\) is equivalent to \(m^T\) and the module \(M\) has a minimal cyclic-filtration consisting of quasi-Gorenstein submodules.
\bigskip
\item[(iii)] \(M\cong \bigoplus\limits_{i=1}^nR/(\lambda_i)\), \(\lambda_i\in R.\)
\bigskip
\item[(iv)] The matrix \(m\) is equivalent to a diagonal matrix.
\end{itemize}
\end{thm}

With regard to the rest of this paper, Section 2 will deal with preliminaries including definitions and motivation, and Section 3 will include the proof of the main result.

\section{Preliminaries}

In this section we will give definitions and properties that will be useful in proving our results in the next section.  More information about the definitions and results which are not new can be found in \cite{AuslanderBridgerSMT, EisenbudCA1995, SerreBook1965}.\newline

Let \(\text{grade}(M)=\inf\{i|\Ext_R^i(M,R)\neq 0\}.\)  The grade of \(M\) is also known as the codimension of the module \(M.\)  In fact, one has \(\dim(M)+\codim(M)=\dim(R)\) in a Cohen-Macaulay ring.  In \cite{Fossum1970} properties of modules of a certain grade is discussed, generalizing results of Auslander.  The focus of \cite{Fossum1970} is on the interplay between the functors \(\Ext^i_R(-,R)\) and the modules.  In \cite{Nagelliaison} quasi-Gorenstein modules are defined to facilitate a generalization of ideal linkage to module linkage. These two papers \cite{Fossum1970, Nagelliaison} provided the inspiration for this paper.

\begin{defn} A finitely generated \(R\)-module \(M\) of finite projective dimension is {\bf quasi-Gorenstein} if the following hold:

\bigskip

\begin{itemize}
\item[(i)] \(\pd(M)=\text{grade}(M)\)
\bigskip
\item[(ii)] \(\Ext^{\pd(M)}_R(M,R)\cong M\)
\end{itemize}

\end{defn}  
 
These are given this name as \(R/I\) is quasi-Gorenstein if and only if \(I\) is a Gorenstein ideal.  Also, Grassi defines \textit{Koszul Modules} \cite{GrassiKoszulM} which are quasi-Gorenstein modules with certain types of free resolutions. Furthermore, a \textit{Gorenstein algebra} is a quasi-Gorenstein module and there is a collection of work about such algebras  \cite{BohningGAlgebrasKoszul,BohningGAlgebrasSzpiro,FlennerVogelUlrichJoins,KleimanUlrichGAlgebras}. \newline

A module $M$ is said to be perfect if $\text{grade}(M)=\pd(M)$.  Quasi-Gorenstein modules are clearly perfect.  In fact, such modules are Cohen-Macaulay (as perfect modules are).  Quasi-Gorenstein modules have many nice properties (aside from those inherited by being Cohen-Macaulay).  For more results and properties of quasi-Gorenstein modules see \cite{YBFossum}.  The following results will be useful in proofs later on.  They follow directly from the definition.  

\begin{prop}[(see \cite{B-H} Exercise 1.4.26)]\label{PDnprop1} Suppose that \(M\) is an \(R\)-module of projective dimension \(n\) which has a projective resolution, \(\mathcal{P}\).  Then \(M\) is quasi-Gorenstein if and only if \(\mathcal{P}\) and \(\mathcal{P}^*=\Hom_R(\mathcal{P},R)\) are homotopy equivalent up to shift, i.e.\ \(\mathcal{P}^*\) is a projective resolution of \(M.\)
\end{prop}

The previous result gives us another method of determining when a module is quasi-Gorenstein.  We only need to know how a free or projective resolution of the module behaves with its dual.  

\begin{cor}\label{PDncor1}
Suppose that \(M\) is an \(R\)-module of projective dimension 1 presented by a full rank \(n\times n\)-matrix \(m.\)  Then \(m\) is equivalent to \(m^T\) if and only if \(M\) is quasi-Gorenstein.
\end{cor}

An example of a module of projective dimension greater than 1 that is quasi-Gorenstein is a complete intersection over regular local ring.  This follows as a free resolution of the complete intersection is the Koszul complex on the regular sequence that generates the ideal, and it is well known that this complex is self-dual.

\begin{prop}\label{PDnproper1}
Suppose that \(\{M_i\}_{i=1}^m\) are all quasi-Gorenstein \(R\)-modules of projective dimension \(n.\)  Then \(\oplus_{i=1}^m M_i\) is a quasi-Gorenstein \(R\)-module of projective dimension \(n.\)
\end{prop}

The following is a special case of a result in \cite{YBFossum}.  We show the proof to illustrate a technique that is useful in gaining information about the Ext modules of \(M\) or a quotient of \(M.\)

\begin{prop}\label{FSR1}
Suppose \(M\) is a quasi-Gorenstein \(R\)-module of projective dimension one and \(M'\)  a submodule of \(M.\)  Then \(\Ext_R^1(M/M',R)\) is isomorphic to a submodule of \(M.\)  Moreover, if \(M/M'\) is quasi-Gorenstein of projective dimension 1 then \(M/M'\) is isomorphic to a submodule of \(M\) and \(\Ass(M/M')\subset \Ass(M).\)
\end{prop}

\begin{proof}: Note that by \cite{Fossum1970} Proposition 3\((a)\) we have that

\[
\text{grade}(M')\geq \text{grade}(M)=1.
\]

The result then follows by dualizing the short exact sequence

\[
\begin{tikzcd}
0 \arrow[r] & M' \arrow[r] & M \arrow[r] & M/M' \arrow[r] & 0
\end{tikzcd}
\]

\noindent to get

\[
\begin{tikzcd}[column sep=3ex]
0 \arrow[r] & \Ext_R^1(M/M',R) \arrow[r] & \Ext_R^1(M,R) \arrow[r] & \Ext_R^1(M',R) \arrow[r] & \Ext_R^2(M/M',R) \arrow[r] & 0
\end{tikzcd}
\]

\noindent which shows that \(\Ext_R^1(M/M',R)\) is isomorphic to a submodule of \(M\cong \Ext^1_R(M,R).\)  Further, if \(M/M'\) is quasi-Gorenstein of projective dimension 1 then the sequence above shows that \(M/M'\) is isomorphic to a submodule of \(M.\)

\end{proof}

This result shows that the associated primes of a module and whether or not it is a quasi-Gorenstein module are intimately related.  It says that if the quotient of a module is quasi-Gorenstein then the associated primes of the quotient are among those of the quasi-Gorenstein module.  One way to guarantee this is to restrict ourselves to such modules.  We can continue this process of looking at quotients with the submodule and obtain a filtration of \(M.\)  Many types of filtrations exist in the literature and we define a type of filtration closely related to those of clean and pretty clean filtrations explored by Herzog and Popescu in \cite{HerPopFF}.  We will denote by \(\mathcal{L}(M)\) the lattice of ideals in \(R\) containing ideals of the form \(\Ann(x)\) for \(x\in M\) under inclusion.  

\begin{defn}\label{cyclicfiltrationdefn}
Let \(M\) be a finitely generated \(R\)-module and 
\[\mathcal{M}:0=M_0\subsetneq M_1\subsetneq M_2\subsetneq M_3\subsetneq \ldots \subsetneq M\]
be an increasing chain of submodules of \(M.\)  We say that \(\mathcal{M}\) is a \textbf{cyclic-filtration} of \(M\) if 
\[M_{i+1}/M_i\cong R/I_{i+1}\qquad \hbox{\rm with}~ I_{i+1} \in \mathcal{L}(M)\quad \hbox{\rm for}~ i=0,1,\ldots\]
\end{defn}

Since \(R\) is Noetherian, the length of a cyclic-filtration is finite.  We will say that a cyclic-filtration is a \textbf{minimal cyclic-filtration} if each module $M_i$ is a minimal submodule of \(M_{i+1}\) such that \(M_{i+1}/M_i\cong R/I_{i+1}\) is maximal for some \(I_{i+1}\in\mathcal{L}(M_{i+1}).\)  In other words, the filtration is minimal at $M_{i+1}$ if the annihilators of the quotients are as small as possible in \(\mathcal{L}(M_{i+1}).\)  In order to see the difference between these consider the module \(\mathbb{Z}/4\mathbb{Z}\) over \(\mathbb{Z}.\)  We have the obvious filtration

\[
0\subsetneq \mathbb{Z}/4\mathbb{Z}
\]

\noindent and the one by using the associated prime of \(\mathbb{Z}/4\mathbb{Z}\) which is

\[
0\subsetneq \mathbb{Z}/2\mathbb{Z} \subsetneq \mathbb{Z}/4\mathbb{Z}.
\]

These are both cyclic-filtrations of \(\mathbb{Z}/4\mathbb{Z},\) but only the first is a \textit{minimal} cyclic-filtration of \(\mathbb{Z}/4\mathbb{Z}.\)  Notice that the sequence

\[
\begin{tikzcd}
0 \arrow[r] & \mathbb{Z}/2\mathbb{Z} \arrow[r] & \mathbb{Z}/4\mathbb{Z} \arrow[r] & \mathbb{Z}/2\mathbb{Z} \arrow[r] & 0
\end{tikzcd}
\]

\noindent is not split.  We will see in the next section that having the filtration be minimal and consisting of quasi-Gorenstein submodules is enough to guarantee such a sequence splits.  Note that the matrix \((1.2)\) has no cyclic-filtration.

\section{Results}

In this section we present and prove the results directly used in proving the main theorem. The following is a key lemma to the proof for Theorem \ref{Rthm1}.

\begin{lem}\label{FSMCFl1}
Let \(R\) be a Noetherian factorial domain.  Suppose \(M\) is a quasi-Gorenstein \(R\)-module of projective dimension one and let \(\mathcal{M}:0\subsetneq M'\subsetneq M\) be a minimal cyclic-filtration of \(M\) where \(M'\) is a perfect \(R\)-module.  Then \(M\cong M'\oplus M/M'.\)
\end{lem}

\begin{proof} Consider the short exact sequence

\[
\begin{tikzcd}
0 \arrow[r] & M' \arrow[r] & M \arrow[r] & M/M' \arrow[r] & 0
\end{tikzcd}
\]

 Where \(M/M'\cong\ R/I\) for some height one ideal \(I\) of \(R.\) Note that \(I\) is principal as \(R\) is a factorial domain, and so \(R/I\) is a quasi-Gorenstein \(R\)-module.  Rewrite the above sequence as 
 
 \[
\begin{tikzcd}
(i): 0 \arrow[r] & M' \arrow[r, "\eta"] & M \arrow[r, "\lambda"] & R/I \arrow[r] & 0
\end{tikzcd}
\]

\noindent and suppose that this sequence is not split exact.  There exists \(m\in M\) such that \(\lambda(m)\) generates \(R/I.\)  Now \(\Ann(m)\subseteq I,\) but as the sequence is not split \(0\neq I\cdot m\subset \eta(M'),\) and so \(\Ann(m)\subsetneq I.\)  Then as \(R/\Ann(m)\) is a submodule of \(M\) we have the short exact sequence

\[
\begin{tikzcd}
(ii): 0 \arrow[r] & R/\Ann(m) \arrow[r, "\gamma"] & M \arrow[r, "\delta"] & K \arrow[r] & 0.
\end{tikzcd}
\]

If we take the dual of \((i)\) we get

\[
\begin{tikzcd}
0 \arrow[r] & R/I \arrow[r, "\lambda^*"] & M \arrow[r, "\eta^*\varphi"] & \Ext_R^1(M',R) \arrow[r] & 0 
\end{tikzcd}
\]

\noindent where \(\varphi\) is the isomorphism between \(M\) and \(\Ext_R^1(M,R).\)  We get a commutative diagram

\[
\begin{tikzcd}[row sep=2ex]
\!\!\!\!\!\!\!\!0 \arrow[r] & R/I \arrow[dd, "\alpha"] \arrow[r, "\lambda^*"] & \Ext_R^1(M,R) \arrow[dd, "\varphi"] \arrow[r, "\eta^*"] & \Ext_R^1(M',R) \arrow[r] \arrow[dd, "\beta"] & 0\\
\!\!\!\!\!\!\!\!(*)~~~~~~~~~~~~~~~~\\
\!\!\!\!\!\!\!\!0 \arrow[r] & R/\Ann(m) \arrow[r, "\gamma"] & M \arrow[r, "\delta"] & K \arrow[r] & 0
\end{tikzcd}
\]

\noindent where \(\alpha\) and \(\beta\) are induced by universal properties as \(\delta\varphi\lambda^*=0.\)  Note that \(\alpha\) is an injection and \(\beta\) is a surjection by the Snake Lemma.  Taking a dual of \((*)\) we get a commutative diagram

\[
\begin{tikzcd}
& & & 0 \arrow[d] \\
& 0 \arrow[d] & 0 \arrow[d] & I/\Ann(m) \arrow[d] \\
0 \arrow[r] & \Ext_R^1(K,R) \arrow[r, "\delta^*"] \arrow[d, "\beta^*"] & \Ext_R^1(M,R) \arrow[r, "\gamma^*"] \arrow[d, "\varphi"] & R/\Ann(m) \arrow[r] \arrow[d, dashed, "\varepsilon"] &  \Ext_R^2(K,R) \arrow[r] & 0\\
0 \arrow[r] & M' \arrow[r, "\eta"] \arrow[d] & M \arrow[r, "\lambda"] \arrow[d] & R/I \arrow[r] \arrow[d] & 0\\
& \coker(\beta^*) \arrow[d] & 0 & 0\\
& 0
\end{tikzcd}
\]

Where the left square is commutative as it is the dual of the right square in \((*).\)  This induces the mapping \(\varepsilon\) which is a surjection.  We claim that \(\Ext_R^2(K,R)=0.\)  \(\Ext_R^2(K,R)\) is the cokernel of the mapping \(\gamma^*.\)  The image of $\gamma^*$ in $R/\Ann(m)$ is isomorphic to $R/I$ as it is the same as $\alpha(R/I)$ in $(*)$ and the image of $M$ through $\lambda$ is $R/I$.  Therefore $\Ext_R^2(K,R)=I/\Ann(m)$, As \(R\) is a factorial domain, if \(\Ext_R^2(K,R)\) is non zero, it only has associated primes of height at least two. This is a contradiction as \(I/\Ann(m)\cong R/(\Ann(m):_RI)\) and \((\Ann(m):_RI)\subset \Ann(m)\) is principal because both \(I\) and \(\Ann(m)\) are principal.  So since \(R\) is factorial, \(R/(\Ann(m):_RI)\) has only associated primes of height one.  Therefore we must have that \(\Ext_R^2(K,R)=0.\)  However, \(\Ext_R^2(K,R)=I/\Ann(M)\) and so $\Ann(M)=I$ a contradiction to the original assumption that \((i)\) is not split.  Therefore $(i)$ must be split.

\end{proof}

\begin{prop}\label{FSMCFp1}
Suppose \(M\) is a quasi-Gorenstein \(R\)-module of projective dimension one and

\[
\mathcal{M}:0=M_0\subsetneq M_1\subsetneq M_2\subsetneq \cdots\subsetneq M_{n-1} \subsetneq M_n=M
\]

\noindent is a minimal cyclic-filtration of \(M\) where \(M_i\) is a quasi-Gorenstein submodule of \(M\) for \(i=1,\ldots, n-1.\)  Then \(M\cong \bigoplus_{i=0}^{n-1} M_{i+1}/M_i.\)
\end{prop}

\begin{proof}  We prove this result by induction on \(n.\)  The case \(n=1\) is trivial as \(M_1=M=M_1/M_0\cong R/I\) for some \(I\in\mathcal{L}(M).\)  The case \(n=2\) is Lemma \ref{FSMCFl1}.  So suppose that \(n>2.\)  We claim that

\[
\mathcal{M}': 0=M_0\subsetneq M_1\subsetneq M_2\subsetneq\cdots\subsetneq M_{n-2}\subsetneq M_{n-1}
\]

\noindent is a minimal cyclic-filtration of \(M_{n-1}\).  Indeed it is a cyclic filtration as \(M_{i+1}/M_i\cong R/I_{i+1}\) with \(I_{i+1}\in\mathcal{L}(M_{i+1})\subset\mathcal{L}(M_{n-1})\) for 
\(i=0,\ldots, n-2.\)  It is minimal as each module \(M_i\) is a minimal submodule of \(M_{i+1}\) with \(M_{i+1}/M_i\cong R/I_{i+1}\) with \(I_{i+1}\in\mathcal{L}(M_{i+1}),\) 
for \(i=1,\ldots ,n-2.\) \newline
 
So by induction \(M_{n-1}\cong \bigoplus_{i=0}^{n-2}M_{i+1}/M_i\cong \bigoplus_{i=0}^{n-2}R/I_{i+1}.\)  Now \(M_{n-1}\) is a minimal submodule of \(M\) with \(M/M_{n-1}\cong R/I\) for \(I\in\mathcal{L}(M).\)  We have the following sequence

\[
\begin{tikzcd}
0 \arrow[r] & \bigoplus_{i=0}^{n-1}R/I_{i+1} \arrow[r, "\alpha"] & M \arrow[r, "\beta"] & R/I \arrow[r] & 0
\end{tikzcd}
\]

Using the same argument as that of Lemma \ref{FSMCFl1} we see that is split and \(M\cong \bigoplus_{i=0}^{n-1} M_{i+1}/M_i.\)

\end{proof}

This says that a quasi-Gorenstein module \(M\) of projective dimension 1 with a minimal cyclic-filtration of quasi-Gorenstein modules has a decomposition of the form 
\(M\cong \bigoplus_{i=1}^{n} R/(\lambda_i)\) for \(\lambda_i\in R\) nonzero.  So if \(M\) has a presentation as

\[
\begin{tikzcd}
0 \arrow[r] & R^n \arrow[r, "\Lambda"] & R^n \arrow[r] & M \arrow[r] & 0
\end{tikzcd}
\]

\noindent with \(\Lambda\) an \(n\times n\)-matrix with elements in \(R,\) then \(\Lambda\) is equivalent to the diagonal matrix \(\text{diag}(\lambda_1,\lambda_2,\ldots,\lambda_n),\) 
and \(M\) is a sum of sequences

\[
\begin{tikzcd}
0 \arrow[r] & R \arrow[r, "\cdot \lambda_i"] & R \arrow[r] & R/(\lambda_i) \arrow[r] & 0
\end{tikzcd}
\]

\noindent for \(i=1,\ldots, n.\) \newline

Now we can prove Theorem \ref{Rthm1}.

\begin{proof} {\em (of Theorem~\ref{Rthm1})} \newline
 \((i)\Leftrightarrow (ii)\) is Corollary \ref{PDncor1}.\newline
 \((iii)\Leftrightarrow(iv)\) is trivial.\newline
 \((i)\Rightarrow(iv)\) is Proposition \ref{FSMCFp1}.\newline

So we are left to prove $(iv)\Rightarrow (i)$.  We know that if $m$ is diagonalizable then $M$ is a quasi-Gorenstein module as $m$ is equivalent to $m^T$.  Next as $(iii)\Leftrightarrow (iv)$ we can take a decomposition of $M\cong \bigoplus_{i=1}^m R/(\lambda_i)$ for $\lambda_i\in R$.  Consider $\mathcal{L}(M)$ and choose $(\lambda_j)$ such that $(\lambda_j)$ is minimal among all $(\lambda_i)$ for $i=1,\ldots, m$.  Note that there may be more than one choice of such ideals.  Let $M_{m-1}=\bigoplus_{i=1, i\neq j}^{m} R/(\lambda_i)$.  Then $M_{m-1}\subsetneq M$ is a piece of a minimal cyclic-filtration of $M$.  It is clear that $M/M_{m-1}=R/(\lambda_j)$ and $M_{m-1}$ is minimal with this property by the choice of $(\lambda_j)$.  We repeat this process for $M_{m-1}$ in $\mathcal{L}(M_{m-1})$ to obtain a minimal submodule $M_{m-2}$ with $M_{m-1}/M_{m-2}= R/(\lambda_k)$ for some $k\in \{1,2,\ldots, m\}\minus\{j\}$.  Continuing in this fashion we obtain $M_1,M_2,\ldots, M_{m-1}$ and a cyclic-filtration of $M$
\begin{align*}
\mathcal{M}: 0=M_0\subsetneq M_1\subsetneq M_2\subsetneq\cdots\subsetneq M_{m-2}\subsetneq M_{m-1}\subsetneq M
\end{align*}

\noindent which is a minimal cyclic-filtration of \(M\) by the choice of each quotient \(M_{i+1}/M_i\) for \(i=0,\ldots ,m-1.\)  Note that each \(M_i\) is a quasi-Gorenstein module by Proposition \ref{PDnproper1} \((a)\) as it is a finite direct sum of quasi-Gorenstein modules \(R/(\lambda_i).\)  Thus \(\mathcal{M}\) is a minimal cyclic-filtration of \(M\) consisting of quasi-Gorenstein submodules of \(M.\)  This proves \((iv)\Rightarrow (i)\) and the Theorem is shown.
\end{proof}



 \bibliographystyle{elsarticle-num}





\end{document}